\documentclass[12pt,a4paper]{article}
\usepackage{bbm}
\usepackage{mathrsfs}
\usepackage{graphicx}
\usepackage{amsmath}
\usepackage{amssymb}
\usepackage{amsfonts}
\usepackage{enumerate}
\usepackage{theorem}

\makeatletter
\def\tank#1{\protected@xdef\@thanks{\@thanks
        \protect\footnotetext[0]{#1}}}
\def\bigfoot{

    \@footnotetext}
\makeatother

\newcommand{\ea}{\end{array}}
\newtheorem{theorem}{Theorem}[section]
\newtheorem{proposition}{Proposition}[section]
\newtheorem{corollary}{Corollary}[section]
\newtheorem{lemma}{Lemma}[section]
\newtheorem{definition}{Definition}[section]
\newtheorem{Rem}{Remark}[section]

{\theorembodyfont{\rmfamily}
}

\newenvironment{proof}{Proof.}

\title
{\bf Random attractor for the 2D stochastic nematic liquid crystals flows with multiplicative noise \thanks{This work was partially
supported by NNSF of China(Grant No. 11401057),   Natural Science Foundation Project of CQ  (Grant No. cstc2016jcyjA0326),
Fundamental Research Funds for the Central Universities(Grant No. 106112015CDJXY100005) and China Scholarship Council (Grant No.:201506055003).} }
\author{
 Guoli Zhou
\thanks{ Chongqing University, P.R. China }
\tank{E-mail:zhouguoli736@126.com.}
}
\begin{document}
\maketitle

\begin{abstract}
Under $non-periodic$ $boundary$ $conditions$, we consider the long-time behavior for stochastic 2D nematic liquid crystals flows with velocity and orientations perturbed by additive noise and multiplicative noise respectively.
The presence of the noises destroys the basic balance law of the  nematic liquid crystals flows, so we can not follow the standard argument to obtain uniform a priori estimates for the stochastic flow  under Dirichlet boundary condition and Numann boundary condition for velocity field and orientation field respectively. To overcome the difficulty our idea is to use some kind of logarithmic energy estimates and It\^{o} formula in some Banach space to obtain the uniform estimates which improve the previous result for the orientation field that grows exponentially w.r.t.time t. In order to study the existence of random attractor, we need to show the solution is a stochastic flow.  But this is not obvious because of the emergence of this kind of multiplicative noise in the orientation field. We give a short proof which is highly non-trivial to show the flow property of the orientation field.  Our idea is to construct several linear stochastic partial differential equations whose scalar valued solutions are stochastic flow, then by discussing the relationship between these scalar  equations and orientation field equation we prove that the each component of orientation field is indeed a stochastic flow  .  Since the global well-posedness is only established for the weak solution, to consider the existence of random attractor, the common method is to derive  uniform a priori estimates in functional space which is more regular than the weak solution space. However, the common method fails because of the ill-posedness of the strong solution. Here, our idea is that by proving the compactness property of the stochastic flow and regularity of the solutions we construct a compact absorbing ball in the weak solution space which implies the existence of the random attractor.  It is\ $\mathbf{the}$\ $\mathbf{first}$ $\mathbf{result}$ for the long-time behavior of stochastic nematic liquid crystals under Dirichlet boundary condition for velocity field and Neumann boundary condition for orientation field.
\end{abstract}



\section{Introduction}
\par
The paper is concerned with the following stochastic hydrodynamical model for the flow of nematic liquid crystals in $\mathbf{D}\times \mathbb{R}^{+}$, where $\mathbf{D}\subset \mathbb{R}^{2}$ is a bounded domain with smooth boundary $\Gamma$.

\begin{eqnarray}
   \mathbf{v}_{t}+[(\mathbf{v}\cdot \nabla)\mathbf{v}-\mu \Delta \mathbf{v}+\nabla p]+\lambda \nabla \cdot (\nabla \mathbf{d}\odot \nabla \mathbf{d})&=&  \dot {W}_{1}, \\
    \nabla\cdot \mathbf{v}(t)&=&0,\\
    \mathbf{d}_{t}+[(\mathbf{v}\cdot \nabla)\mathbf{d}]- \gamma (\Delta \mathbf{d}- f(\mathbf{d}))&=&(\mathbf{d}\times h)\circ \dot{W}_{2}.
\end{eqnarray}
The unknowns for the 2D stochastic  hydrodynamical model are the fluid velocity field $\mathbf{v}=(\mathbf{v}^{1}, \mathbf{v}^{2})\in \mathbb{R}^{2},$ the averaged macroscopic/continuum molecular orientations $ \mathbf{d}=({d}^{1}, {d}^{2}, d^{3} ) \in \mathbb{R}^{3}$ and the scalar function $p(x,t)$ representing the pressure (including both the hydrostatic and the induced elastic part from the orientation field). The positive constants $\nu, \lambda $ and $\gamma$ stand for viscosity, the competition  between kinetic energy and potential energy, and macroscopic elastic relaxation time (Debroah number) for the molecular orientation field.   ${W}_{1}$ is a standard Wiener process in $\mathbf{H}$ defined below and has the form of
\begin{eqnarray*}
{W}_{1}(t):= \sum_{i=1}^{\infty} \lambda_{i}^{\frac{1}{2}}e_{i}B_{i}(t),
\end{eqnarray*}
where ${(} ( B_{i}(t) )_{t\in \mathbb{R}}  {)}_{i\in \mathbb{N}}$ be
a sequence of one-dimensional, independent, identically distributed Brownian
motions defined on  the complete probability space $(\Omega, \mathcal{F}, \mathbb{P} )$, $(e_{i})_{i\in \mathbb{N}}$ is an orthonormal basis in $\mathbf{H},$ $(\lambda_{i})_{i\in \mathbb{N}}$ is a convergent sequence of positive numbers which assure $W_{1}$ is a standard Wiener process in $\mathbf{H}$. ${W}_{2}$ is a standard real-valued Brownian motion with $(\mathbf{d}\times h)\circ \dot{W}_{2}$ understood in the stratonovich sense. $h=(h_{1}, h_{2}, h_{3})\in \mathbb{R}^{3},$ $h_{i}$ is constant, $i=1, 2, 3 .$ Here $f:\mathbb{R}^{3}\rightarrow \mathbb{R}^{3} $ is a general polynomial function whose details will be given later.
The symbol $\nabla \mathbf{d} \odot \nabla \mathbf{d}$ denote the $2\times 2$ matrix whose $(i,j)$-th entry is given by
\begin{eqnarray*}
[\nabla \mathbf{d}\odot \nabla \mathbf{d}]_{i,j}=\sum^3_{k=1}\partial_{x_i}\mathbf{d}^{(k)}\partial_{x_j}\mathbf{d}^{(k)},\quad i,j=1,2.
\end{eqnarray*}
In this paper, we consider the following initial boundary conditions for the stochastic nematic liquid crystals equations. Boundary conditions
\begin{eqnarray}
\mathbf{v}(t,x)=0,\ \ \ \  \frac{\partial\mathbf{d}(x,t)}{\partial \mathbf{n}}=0,\ \ \ \mathrm{for}\ (x,t)\in \Gamma\times \mathbb{R}^{+}.
\end{eqnarray}
Initial conditions
\begin{eqnarray}
\mathbf{v}|_{t=0}=\mathbf{v}_{0}(x)\ \mathrm{with}\ \nabla\cdot \mathbf{v}_{0}=0,\ \ \ \  \mathbf{d}|_{t=0}=\mathbf{d}_{0}(x),\ \ \ \mathrm{for}\ x\in \mathbf{D},
\end{eqnarray}
where $\mathbf{n}$ is the outward unite normal vector to $\Gamma.$
In $\cite{Lin} $, F.H. Lin proposed a corresponding deterministic model of $(1.1)-(1.3)$ as a simplified system of the original Ericksen-Leslie system $(see\ \cite{E, Le})$. By Ericksen-Leslie's hydrodynamical theory of the liquid crystals, the simplified system describing the orientation as well as the macroscopic motion reads as follows
\begin{eqnarray}
   d\mathbf{v}+[(\mathbf{v}\cdot \nabla)\mathbf{v}-\mu \Delta \mathbf{v}+\nabla p]dt&=&\lambda \nabla \cdot (\nabla \mathbf{d}\odot \nabla \mathbf{d})dt, \\
    \nabla\cdot \mathbf{v}(t)&=&0,\\
    \partial_{t}\mathbf{d}+[(\mathbf{v}\cdot \nabla)\mathbf{d}]&=&\gamma (\Delta \mathbf{d}(t)+ |\nabla \mathbf{d}|^{2}\mathbf{d}),\ \ |\mathbf{d}|=1.
\end{eqnarray}
 In order to avoid the nonlinear gradient in $(1.8)$, usually one uses the Ginzburg-Landau approximation to relax the constraint $\mathbf{d}=1 .$ The corresponding approximate energy is
 \begin{eqnarray}
  \int_{\mathbf{D}} [\frac{1}{2}|\nabla \mathbf{d}|^{2}+\frac{1}{4\eta^{2}}(|\mathbf{d}|^{2}-1)^{2}]dx
\end{eqnarray}
 where $\eta$ is a positive constant. Then one arrives at the approximation system $(1.6)-(1.8)$ with $f(\mathbf{d})$ and $F(\mathbf{d})$ given by
\begin{eqnarray}
f(\mathbf{d})=\frac{1}{\eta^{2}}(|\mathbf{d}|^{2}-1)\mathbf{d}\ \ \mathrm{and}\ \ F(\mathbf{d})= \frac{1}{4\eta^{2}}(|\mathbf{d}|^{2}-1)^{2}.
\end{eqnarray}
In this work, we consider a more general polynomial function $f(\mathbf{d})$ which contains as a special case the $(1.10)$.
We define a function $ \tilde{f}:[0, \infty)\rightarrow \mathbb{R}$ by
 \begin{eqnarray}
\tilde{f}(x)=\sum_{k=0}^{N}a_{k}x^{k},  \ \ \   x\in \mathbb{R}_{+},
\end{eqnarray}
where $  a_{N}>0$ and $a_{k} \in \mathbb{R},k=0,1,2,...,N-1.$
Let $f:\mathbb{R}^{3}\rightarrow  \mathbb{R}^{3}$ given by
 \begin{eqnarray}
f(\mathbf{d})= \tilde{f}(|\mathbf{d}|^{2})\mathbf{d}.
\end{eqnarray}
Denote by ${F}: \mathbb{R}^{3} \rightarrow \mathbb{R}$ the Fr$\acute{e}$chet differentiable map such that for any $\mathbf{d}\in \mathbb{R}^{3}$ and $\xi\in \mathbb{R}^{3}$
\begin{eqnarray}
{F}'(\mathbf{d})[\xi]=f(\mathbf{d})\cdot \xi.
\end{eqnarray}
Set $\tilde{F}$ to be an antiderivative of $\tilde{f}$ such that $\tilde {F}(0)=0.$ Then
$$\tilde{F}(x)= \sum_{k=0}^{N}\frac{a_{k}}{k+1}x^{k+1},  \ \ \   x\in \mathbb{R}_{+}.$$
\par
The Ericksen--Leslie system is well suited for describing many special flows for the
materials, especially for those with small molecules, and is widely accepted in the
engineering and mathematical communities studying liquid crystals. System (1.1)--
(1.3) with $f(\mathbf{d})$ given by $(1.10)$ can be possibly viewed as the simplest mathematical model, which keeps the
most important mathematical structure as well as most of the essential difficulties
of the original Ericksen--Leslie system (see $\cite{LL}$). This deterministic system with Dirichlet
boundary conditions has been studied in a series of work not only theoretically
(see $\cite{LL}, \cite{LL1} $) but also numerically (see $\cite{LW}, \cite{LW1} $).

\par
The introduction of
stochastic processes in nematic liquid crystals flows is aimed at accounting for a number of
uncertainties and errors:\\
(1)The state of the nematic liquid crystals  is strongly dependent on the state of the environment. In natural systems external noise is often quite large. At an instability point the systern is sensitive even to infinitesimally small perturbations and the role of external noise has to be investigated in the vicinity of a transition point. And experimental investigation also showed that, in the case of electrohydrodynamic instabilities, the average value of the voltage necessary for the transition is shifted to higher and higher values as the intensity of the external noise is increased  i.e., the amplitude of the voltage fluctuations, is increased. Further study showed that the average voltage necessary to induce the transition to turbulent behavior, increases with the variance of the voltage fluctuations (see $\cite{HDLC}$ ). For more details one can also refer to $\cite{BHR1, BHR2, BMAA, HL, SM, SS}.$\\
(2)Rheological predictions of the behavior of complex fluids like these, often start with the derivation of macroscopic, approximate equations for quantities of interest using various  closure approximations. The difficulty in obtaining accurate closures has motivated the extensive, in recent years, use of direct simulations, either of the PDE governing the orientation distribution function, or of the equivalent stochastic differential equation, via ¡°Brownian dynamics¡± simulations. The latter have the advantage that they are amenable to use with models with many internal degrees of freedom (as opposed to the PDE approach in which the ¡°curse of dimensionality¡± precludes realistic computation). For more details one can see $\cite{G, JDP, LO, SK1,SK2}.$

\par
Despite the developments in the deterministic case, the theory for the
stochastic  nematic liquid crystals remains underdeveloped. To the best of our knowledge, there are few works on the stochastic nematic liquid crystals.  In the papers $\cite{BHR1, BHR2}$, Z.Brzezniak, E.Hausenblas and P.Razafimandimby have considered the model perturbed by multiplicative Gaussian noise and have proved the global well-posedness for the weak solution and strong solution in 2-D case. When the noise is jump and the dimension is two,  Z.Brzezniak, U.Manna and A.A. Panda in $\cite{BMAA}$ obtained the same result as the case of Gaussian noise. A weak martingale solution is also established for the three dimensional stochastic nematic liquid crystals with jump noise in $\cite{BMAA}$.

\par
One natural problem arising from this global existence result is the dynamical behavior of the 2D stochastic system.

\par
Under the periodic boundary conditions or the assumption $\mathbf{d}(t,x)=\mathbf{d}_{0}(x)$ for $(x,t)\in \Gamma\times \mathbb{R}^{+} ,$  the existence of global attractor is established for the ${deterministic}\ {model}$(see $\cite{GW, Sh} $). However, there is no published result about the existence of random attractor for the stochastic nematic liquid crystals flows. One reason is the absence of the $basic$ $balance$ $law$ which results in the failure of applying the method of deterministic model to the stochastic model. The other important reason is due to the different boundary conditions   between the deterministic model and stochastic model. As we see if the boundary condition is given by $(1.4)$, the following equality is not ture
\[
\langle \Delta \mathbf{d}-f(\mathbf{d}), \Delta (\Delta \mathbf{d}-f(\mathbf{d}) )     \rangle= -| \nabla(\Delta \mathbf{d}-f(\mathbf{d})|_{2}^{2},
\]
where $\langle ,  \rangle$ and $|\cdot|_{2}$ denote the inner product and norm in $(L^{2}(\mathbf{D}))^{3}.$ Therefore, we can not obtain the global existence of the strong solution in partial differential equations sense.

\par
In this article, we firstly improve the bounds for the solutions to $(1.1)-(1.5).$ These bounds are uniform with respect to present time and initial time (see Lemma 3.1). These estimates of the uniform boundedness improve previous bounds obtained in $\cite{BHR1, BHR2, BMAA},$ in which the bounds of $\mathbf{d}$ grow exponentially with respect to present time or initial time. In obtaining these time-uniform $a$ $priori$ estimates for orientation field $\mathbf{d}$ (for example in space $(L^{2}(D))^{3}$),  the power of $\mathbf{d}$ from the nonlinear $f(\mathbf{d})$ will be much bigger  (see $(1.12)$) than two. Using the Gronwall inequality (a standard argument) we will obtain that the solutions have exponential growth which is not sufficient to ensure the existence of random attractor. To overcome the difficulty, our idea is that after applying  It\^{o} formula in Banach space $(L^{4N+2}(\mathbf{D}))^{3}, N>1,$ to $\mathbf{d}$, we try to take advantage of the property of the logarithmic function to reduce the power from nonlinear term.
Roughly speaking, for a positive $f(t), t\in \mathbb{R}_{+},$ if we want to consider it's uniform boundedness with respect to time $t,$  we just need to estimate $\ln (1+f(t)).$ If  $\ln (1+f(t))$ is uniformly bounded with respect to $t,$ so is $f(t).$ Using this new technique
we obtain the uniform estimates for orientation field $\mathbf{d}$ ( see Lemma 3.1) which opens a way  to study the long-time behavior of stochastic nematic liquid crystals.

\par
To show the existence of random attractor, another problem needed to be addressed is to prove the solution $(\mathbf{v}, \mathbf{d})$ to $(1.1)-(1.5)$ is indeed a stochastic flow. The difficulty lies in $(1.3),$ we should show $\mathbf{d}$ to $(1.3)$ is a stochastic flow. But for this equation, we can not follow the method in $\cite{CF}$ of constructing a equivalent vector valued partial differential equation with random coefficients to prove that the  stochastic orientation field is a stochastic flow.
 To overcome the difficulty we construct a linear stochastic partial differential equations (SPDE) with linear  Stratonovich multiplicative noise whose scalar valued solution is a stochastic flow. Then for different $h$, we will obtain different linear and scalar valued SPDEs whose solutions are stochastic flow. Taking advantage of the relationship between these linear SPDEs and $(1.3)$, we can prove one component of the solution $\mathbf{d}$ to $(1.3)$ is indeed a stochastic flow. Then repeating the arguments, we can infer that each component of $\mathbf{d}$ is a stochastic flow which implies the flow property of stochastic field $\mathbf{d}$. For the detail we can refer to Proposition 3.2 and Remark 3.3.

\par
Our main goal of this article is to show the existence of random attractor in the solution space $\mathbf{H}\times \mathbb{H}^{1}$.
As we know, the sufficient condition for ensuring the existence of random attractor in $\mathbf{H}\times \mathbb{H}^{1}$ is to
obtain an absorbing ball which is compact in $\mathbf{H}\times \mathbb{H}^{1}$. The common method is to derive  uniform a priori estimates in the functional space $\mathbf{V}\times \mathbb{H}^{2} $ which is the strong solution space. However, the global existence of strong solution is unavailable due to the Neumann boundary condition. Here, we use a compactness arguments of the stochastic flow and regularity of the solutions to construct a compact absorbing ball in the function space $ \mathbf{H}\times \mathbb{H}^{1}$.  We complete the proof of the existence of random attractor by four steps. Firstly, using the Lemma 3.1, we obtain the absorbing ball in the weak solution $\mathbf{H}\times \mathbb{H}^{1}$ (see Proposition 3.1). Secondly, we will verify the two a priori estimates of Aubin-Lions compact lemma to obtain a convergent subsequence of $(\mathbf{v}, \mathbf{d})$  which converges almost everywhere with respect to time $t\in [s, T], -\infty<s<T<\infty$ (see Proposition 3.3 and Proposition 3.4).
Thirdly, in order to show the solution operator is a stochastic dynamical system  we show that
$
 \mathbf{v}\in C([0,T];\mathbf{H})\ \mathrm{and} \  \mathbf{d}\in C([0,T];\mathbb{H}^{1}),
$
which improves the regularity of the solution $(\mathbf{v},\mathbf{d})$ obtained in $\cite{BHR1}.$
For the details we can refer to Corollary 3.1.
 Using the regularity of the solutions to $(1.1)-(1.5)$ and the Aubin-Lions Lemma we prove in Proposition 3.5 that the solution operators are almost surely compact in $\mathbf{H}\times \mathbb{H}^{1}$ for all $(t,s)$ satisfying $ -\infty<s<t<\infty$.  Finally,  in Proposition 3.6  using the compact solution operator to act on the the absorbing ball yields a new set which is compact and absorbing in $\mathbf{H}\times \mathbb{H}^{1}$.  The existence of the random attractor in the weak solution space $\mathbf{H}\times \mathbb{H}^{1}$ follows directly from Proposition 3.6.

\par
This article illustrates  some advantage of our method over the common method. For the present model, following the common method to prove the existence of random attractor, we need to obtain uniform a priori estimates in a function space $\mathbf{V}\times \mathbb{H}^{2}.$ As we see it is very difficult, i.e., the uniform a priori estimates for  $(1.1)-(1.5)$ in more regular function space than the solution space is  not available.   Our method here avoid doing estimates in function spaces $\mathbf{V}\times \mathbb{H}^{2}$, but prove the compact absorbing ball in $\mathbf{H}\times \mathbb{H}^{1}$ indeed  exists .

\par
The remaining of this paper is organized as follows. In section $2,$ we state some preliminaries and recall some results.  The existence of random attractor is presented in section $3$. As usual, constants
$C$ may change from one line to the next, unless, we give a special declaration
; we denote by $C(a)$ a constant which depends on some parameter $a.$

\section{Preliminaries}
For $1\leq p\leq \infty,$ let $L^{p}(\mathbf{D})$ be the usual Lebesgue spaces with the norm $|\cdot|_{p}$ .  For a positive integer
$m,$ we denote by $(H^{m,p}(\mathbf{D}), \|\cdot\|_{m,p})$  the usual Sobolev spaces, see($\cite{ARA}$). When $p=2,$ we denote by
$(H^{m}(\mathbf{D}), \|\cdot\|_{m})$ with inner product $\langle, \rangle_{H^{m}}$. Let
\begin{eqnarray*}
\mathcal{V}=\{\mathbf{\upsilon}\in (C_{0}^{\infty}(\mho))^{2}&:& \nabla \cdot \mathbf{\upsilon} =0 \}.
\end{eqnarray*}

We denote by $\mathbf{H}, \mathbf{V}$ and $\mathbf{H}^2$ be the closure spaces of $\mathcal{V}$ in $(L^{2}(\mathbf{D}))^{2}, (H^{1}(\mathbf{D}))^{2}$ and $(H^{2}(\mathbf{D}))^{2}$ respectively. And set $|\cdot|_{2}$ and $\langle , \rangle$ to be the norm and inner product of  $\mathbf{H}$ respectively.
The notation $\langle , \rangle$ is also used to denote the inner product in $(L^{2}(\mathbf{D}))^{2}.$ By the Poincar\'{e} inequality,  there exists a constant $c$ such that for any $\mathbf{v}\in \mathbf{V}$ we have $\|\mathbf{v}\|_{1}\leq c|\nabla \mathbf{v}|_{2}$. Without confusion, we let $\|\cdot\|_{1}$ and $\langle, \rangle_{\mathbf{V}}$ stand for the  norm and the inner product in $\mathbf{V}$ respectively, where $\langle, \rangle_{\mathbf{V}}$ is defined by
\begin{eqnarray*}
\langle \mathbf{v}_{1},  \mathbf{v}_{2}\rangle_{\mathbf{V}}:=\int_{\mathbf{D}}\nabla \mathbf{\upsilon}_{1}\cdot \nabla \mathbf{\upsilon}_{2}d\mathbf{D},\ \ \mathrm{for}\  \mathbf{v}_{1}, \mathbf{v}_{2}\in \mathbf{V}.
\end{eqnarray*}
Denote by $\mathbf{V}'$ the dual space of $\mathbf{V}$. And define the linear operator $A_{1}:\mathbf{V} \mapsto \mathbf{V}',$ as the following:
\begin{eqnarray*}
 \langle A_{1}\mathbf{v}_{1},    \mathbf{v}_{2} \rangle=\langle \mathbf{v}_{1},  \mathbf{ v}_{2} \rangle_{\mathbf{V }},\ \   \mathrm{for} \ \mathbf{v}_{1},\mathbf{v}_{2}\in  \mathbf{V}.
\end{eqnarray*}
Since the operator $A_{1}$ is positive selfadjoint with compact resolvent,
by the classical spectral theorems there exists a sequence $\{\alpha_{j}\}_{j\in \mathbb{N}} $ of eigenvalues of $A_{1}$ such that
$$0<\alpha_{1}\leq \alpha_{2}\leq \cdots, \ \ \alpha_{j}\rightarrow\infty$$
corresponding to the eigenvectors $e_{j}.$ Assume
\begin{eqnarray}
 \sum_{i=1}^{\infty} \lambda_{i}\alpha_{i}^{2}<\infty.
\end{eqnarray}
For arbitrary constant $T>0$ and $j\in \mathbb{N} ,$ we define
$$\mathbf{z}^{j}(t)= \sum_{n=1}^{j}\sqrt{\lambda_{n}}\int_{0}^{t}e^{-A_{1}(t-s)}e_{n}dB_{n}(s),\ \ t\in[0,T]  $$
and
$$\mathbf{z}(t)= \sum_{n=1}^{\infty}\sqrt{\lambda_{n}}\int_{0}^{t}e^{-A_{1}(t-s)}e_{n}dB_{n}(s),\ \ t\in[0,T]  .$$
Obviously,
$$\mathbf{z}^{j}(w)\in C([0,T];\mathbf{H}^{2}),\ \mathbb{P}-a.e.\ \omega\in \Omega . $$
For $k\in \mathbb{N}$ and $k> j, $ In view of an infinite dimensional
version of Burkholder-Davis-Gundy type of inequality for stochastic convolutions (see Theorem 1.2.6 in $\cite{DZ, Liu}$ and references therein), we have
\begin{eqnarray*}
&&E\sup\limits_{t\in [0,T] }\|A_{1}(\mathbf{z}^{j}-\mathbf{z}^{k}) \|^{2}_{L^{2}} \\
&\leq&CT\sum_{n=j+1}^{k} \lambda_{n}\alpha_{n}^{2}\rightarrow 0,\ \ as\ j\rightarrow\infty.
\end{eqnarray*}
Therefore
\begin{eqnarray}
\mathbf{z}(w)\in  C([0,T]; \mathbf{H}^{2}),\ \mathbb{P}-a.e.\ \omega\in \Omega.
\end{eqnarray}
Let $\mathbb{ H}^{m}=({H}^{m}(\mathbf{D}))^{3},m=0, 1,2,.... $ When $m=0,$ set $\mathbb{H}=\mathbb{H}^{0}= (L^{2}(\mathbf{D}))^{3} $ for simplicity. Denote the dual space of $\mathbb{ H}^{m}$ by $\mathbb{ H}^{-m} .$ Then, similarly, we define the linear operator $A_{2}:\mathbb{H}^{1} \mapsto \mathbb{H}^{-1} $ as
\begin{eqnarray*}
 \langle A_{2}\mathbf{d}_{1},    \mathbf{d}_{2} \rangle=\langle \mathbf{d}_{1},  \mathbf{ d}_{2} \rangle_{H^{1}},\ \   \mathrm{for} \ \mathbf{d}_{1},\mathbf{d}_{2}\in \mathbb{H}^{1} .
\end{eqnarray*}

Let $D(A_{1}):=\{\eta\in \mathbf{H}, A_{1}\eta\in \mathbb{H }   \}$ and $D(A_{2}):= \{\theta\in \mathbb{H}^{1}, A_{2}\eta\in \mathbb{H}   \}.$ Because $A_{1}^{-1}$ and $A_{2}^{-1}$ are self-adjoint compact operators in $\mathbf{H}$ and $\mathbb{H}$ respectively, thanks to the classic spectral theory, we can define the power $A_{i}^{s}$ for any $s\in \mathbb{R}.$ Then $D(A_{i})'= D(A_{i}^{-1})$ is the dual space of $D(A_{i})$. Furthermore, we have the compact embedding relationship
\begin{eqnarray*}
D(A_{1})\subset \mathbf{V} \subset \mathbf{H} \subset \mathbf{V}'\subset D(A_{1})',
\end{eqnarray*}
and
\begin{eqnarray*}
\langle \cdot,  \cdot \rangle_{\mathbf{V}}=\langle A_{1}\cdot, \cdot    \rangle= \langle A_{1}^{\frac{1}{2}}\cdot, A_{1}^{\frac{1}{2}}\cdot\rangle.
\end{eqnarray*}
Similarly,
\begin{eqnarray*}
D(A_{2})\subset \mathbb{H}^{1} \subset \mathbb{H} \subset \mathbb{H}^{-1}\subset D(A_{2})',
\end{eqnarray*}
and
\begin{eqnarray*}
\langle \cdot,  \cdot \rangle_{H^{1}}=\langle A_{2}\cdot, \cdot    \rangle= \langle A_{2}^{\frac{1}{2}}\cdot, A_{2}^{\frac{1}{2}}\cdot\rangle.
\end{eqnarray*}

\begin{definition}
We say a continuous $\mathbf{H}\times \mathbb{H}^{1}$ valued $(\mathcal{F}_{t}  )=(\sigma(W(s), s\in [0,t]) )$ adapted random field $(\mathbf{v}(.,t), \mathbf{B}(.,t) )_{t\in [0,T]}$ defined on $(\Omega, \mathcal{F}, \mathbb{P})$ is a weak solution to problem $(1.1)-(1.5)$ if for $(\mathbf{v}_{0}, \mathbf{d}_{0})\in \mathbf{H}\times \mathbb{H}^{1}$ the following conditions hold:
\begin{eqnarray*}
&& u\in C([0,T];\mathbf{H})\cap L^{2}([0,T]; \mathbf{V}),\\
&& \theta\in C([0,T];\mathbb{H}^{1})\cap L^{2}([0,T]; \mathbb{H}^{2}),
\end{eqnarray*}
and the integral relation
\begin{eqnarray*}
\langle \mathbf{v}(t), v   \rangle &+&\int_{0}^{t}\langle  A_{1}\mathbf{v}(s),  v  \rangle ds+\int_{0}^{t}\langle \mathbf{v}(s)\cdot\nabla \mathbf{v}(s),  v  \rangle ds\\
&&+ \int_{0}^{t}\langle\nabla \cdot (\nabla \mathbf{d}(s)\odot \nabla \mathbf{d}(s)), v\rangle ds=\langle \mathbf{v}_{0}, v \rangle+\langle W_{1}(t), v   \rangle ,\\
\langle \mathbf{d}(t),  d  \rangle&+&\int_{0}^{t}\langle  A_{2}\mathbf{d}(s),  d  \rangle ds+\int_{0}^{t}\langle \mathbf{v}(s)\cdot\nabla \mathbf{d}(s), d\rangle ds\\
&&=\langle \mathbf{d}_{0}, d   \rangle-\int_{0}^{t}\langle f(\mathbf{d}(s)), d\rangle ds+ \frac{1}{2}\int_{0}^{t}\langle \mathbf{d}\times h\times h, d    \rangle+\int_{0}^{t} \langle \mathbf{d} , d   \rangle dW_{2}(s),
\end{eqnarray*}
hold $a.s.$ for all $t\in [0,T]$ and $(v,d)\in \mathbf{V}\times \mathbb{H}.$
\end{definition}

To prove the existence of random attractor for stochastic liquid crystals flows, we need the
following result concerning global well-posedness of $(1.1)-(1.5)$. For the proof, one can follow the argument as in $\cite{BHR1}$ or $\cite{BMAA}$ with minor revisions.
\begin{theorem}
Let $(\mathbf{v}_{0}, \mathbf{d}_{0}) \in \mathbf{H}\times \mathbb{H}^{1} .$  Assume conditions $(2.14)$
hold. Then there exists a unique weak solution $(\mathbf{v}, \mathbf{d} )$ of the system $(1.1)-(1.5)$ on the interval $[0, T],$
which is Lipschitz continuous  with respect to the initial data  in $\mathbf{H} \times \mathbb{H}^{1}.$
\end{theorem}
\begin{Rem}
 In Theorem 3.2 of $\cite{BHR1}$, it is proved that the weak solution $(\mathbf{v}, \mathbf{d})$ to $(1.1)-(1.5)$ satisfies
\begin{eqnarray*}
\mathbf{v}\in C([0,T]; \mathbb{V}^{-\beta})\  \mathrm{and}\ \mathbf{d}\in C([0,T];\mathbb{X}_{\beta-\frac{1}{2}}),\  a.s.,\  \beta \in (0, \frac{1}{2}).
\end{eqnarray*}
But this kind of regularity can not make the solution operator be a stochastic dynamical system. In Corollary 3.1, we prove
\begin{eqnarray*}
&& \mathbf{v}\in C([0,T];\mathbf{H}),\ \mathrm{and}\ \mathbf{d}\in C([0,T];\mathbb{H}^{1}),\ a.s..
\end{eqnarray*}
\end{Rem}
\par
In the following, we recall the notations and results in stochastic dynamical systems which will be use to prove the main results of this article.
\par
Let $(X, d)$ be a polish space and $(\tilde{\Omega}, \tilde{\mathcal{F}}, \tilde{\mathbb{P}}  )$ be a probability space, where $ \tilde{\Omega}$ is the two -sided Wiener space $C_{0}(\mathbb{R}; X )$ of continuous functions with values in $X$, equal to $0$ at $t=0$. We consider a family of mappings
$S(t,s;\omega):X\rightarrow X,\ \ -\infty<s \leq t< \infty,$ parametrized by $\omega\in \tilde{\Omega},$ satisfying for $ \tilde{\mathbb{P}}$-$a.e.\ \omega$ the following properties (i)-(iv):
\par
(i)$\ \ S(t,r;\omega)S(r,s;\omega)x= S(t,s;\omega)x$ for all $s\leq r\leq t$ and $x\in X;$
\par
(ii)$\ \ S(t,s;\omega)$ is continuous in $X,$ for all $s\leq t;$
\par
(iii)\ \  for all $s<t$ and $x\in X$, the mapping
\[
\omega\mapsto S(t,s;\omega)x
\]
\ \ \ \ \ \ \ \ \ \ \ \ is measurable from $(\tilde{\Omega},\tilde{\mathcal{F}})$ to $(X, \mathcal{B}(X )  )$ where $\mathcal{B}(X ) $ is the Borel-$\sigma$- algebra of $ X$;
\par
(iv)\ \ for all $t, x\in X,$ the mapping $s\mapsto S(t,s;\omega)$ is right continuous at any point.
\par
A set valued map $K: \tilde{\Omega}\rightarrow 2^{X}$ taking values in the closed subsets of $X$ is said to be measurable if for each $x\in X$ the map
$\omega\mapsto d(x, K(\omega))$ is measurable, where $d(A,B)=\sup\{\inf\{d(x,y):y\in B \}:x\in A \}$ for $A,B \in 2^{X}, A,B\neq \emptyset;$
and $d(x,B)=d(\{x\},B).$ Since $ d(A,B)=0$ if and only if $A\subset B, d$ is not a metric. A closed set valued measurable map $K:\tilde{\Omega}\rightarrow 2^{X}$ is named a random closed set.
\par
Given $t\in \mathbb{R}$ and $\omega\in \tilde{\Omega}, K(t,\omega)\subset X$ is called an attracting set at time $t$ if , for all bounded sets $B\subset X,$
\[
d(S(t,s;\omega)B, K(t,\omega) )\rightarrow 0,\ \ provided\ s\rightarrow -\infty.
\]
Moreover, if for all bounded sets $B\subset X,$ there exists $t_{B}(\omega)$ such that for all $s\leq t_{B}(\omega)$
\[
S(t,s;\omega)B\subset K(t,\omega),
\]
we say $K(t,\omega) $ is an absorbing set at time $t.$

Let $\{\vartheta_{t}:\tilde{\Omega}\rightarrow \tilde{\Omega}   \}, t\in T, T=\mathbb{R},$ be a family of measure preserving transformations of the probability space $(\tilde{\Omega}, \tilde{\mathcal{F}},\tilde{ \mathbb{P}} )$ such that for all $s< t$ and $\omega\in \tilde{\Omega}$
\par
(a) $(t,\omega)\rightarrow \vartheta_{t}\omega$ is measurable;
\par
(b) $\vartheta_{t}(\omega)(s)=\omega(t+s)-\omega(t)$;
\par
(c) $S(t,s;\omega)x=S(t-s,0;\vartheta_{s}\omega)x.$\\
Thus $(\vartheta_{t} )_{t\in T}$ is a flow, and
$((\tilde{\Omega}, \tilde{\mathcal{F}},\tilde{ \mathbb{P}} ), (\vartheta_{t} )_{t\in T} )$ is a measurable dynamical system.

\begin{definition}
Given a bounded set $B\subset X$, the set
\begin{eqnarray*}
\mathcal{A}(B,t,\omega)=\bigcap\limits_{T\leq t}\overline{\bigcup\limits_{s\leq T}S(t, s,\omega)B}
\end{eqnarray*}
is said to be the $\Omega$-limit set of $B$ at time $t$. Obviously, if we denote $\mathcal{A}(B,0,\omega)=\mathcal{A}(B,\omega),$ we have
$\mathcal{A}(B,t,\omega)=\mathcal{A}(B,\vartheta_{t}\omega).$
\end{definition}
We may identify
\begin{eqnarray*}
\mathcal{A}(B,t,\omega)=\{x\in X &:  &\mathrm{there}\ \mathrm{exists}\ s_{n}\rightarrow -\infty\ \mathrm{and}\ x_{n}\in B\nonumber\\
 &&\mathrm{such}\ \mathrm{that}\ \lim\limits_{n\rightarrow\infty}S(t,s_{n},\omega)x_{n}=x\}.
\end{eqnarray*}
Furthermore, if there exists a compact attracting set $K(t,\omega)$ at time $t,$ it is not difficult to check that $\mathcal{A}(B,t,\omega)$ is a nonempty compact subset of $X$ and $\mathcal{A}(B,t,\omega)\subset K(t,\omega). $
\begin{definition}
If, for all $t\in \mathbb{R}$ and $\omega\in \tilde{\Omega},$ the random closed set $\omega\rightarrow \mathcal{A}(t,\omega)$ satisfying the following properties:
\par
(1) $\mathcal{A}(t,\omega)$ is a nonempty compact subset of $X,$
\par
(2) $\mathcal{A}(t,\omega)$ is the minimal closed attracting set,
i.e., if $\tilde{\mathcal{A}}(t,\omega)$ is another closed attracting set, then $\mathcal{A}(t,\omega)\subset \tilde{\mathcal{A}}(t,\omega),$
\par
(3) it is invariant,  in the sense that, for all $s\leq t,$
 \[
 S (t,s;\omega)\mathcal{A}(s,\omega)=\mathcal{A}(t,\omega),
 \]
$\mathcal{A}(t,\omega)$  is called the random attractor.
\end{definition}
Let
$$\mathcal{A}(\omega)=\mathcal{A}(0,\omega). $$
Then the invariance property writes
$$S(t,s;\omega)\mathcal{A}(\vartheta_{s}\omega)=\mathcal{A }(\vartheta_{t}\omega).$$
To prove the existence of the random attractor, we will use the following sufficient condition given in $\cite{CDF} $.
For the convenience of reference, we cite it here.
\begin{theorem}
Let $(S(t, s; \omega))_{t\geq s, \omega\in \tilde{\Omega}}$ be a stochastic dynamical system
satisfying $\mathrm{(i)}, \mathrm{(ii)}, \mathrm{(iii)}$ and $\mathrm{(iv)}$. Assume that there exists a group $ \vartheta_{t}, t\in \mathbb{R},$ of measure preserving mappings such that condition $(c)$ holds and that, for $\tilde{\mathbb{P}}$-a.e.\ $\omega,$ there exists a compact attracting set $K(\omega)$ at time $0.$ For $\tilde{\mathbb{P}} $-a.e.\ $\omega,$ we set
$$\mathcal{A}(\omega)=\overline{\bigcup_{B\subset X}\mathcal{A}(B,\omega) } $$
where the union is taken over all the bounded subsets of $X$. Then we have for $\tilde{\mathbb{P}}  $-a.e.\ $\omega\in \tilde{\Omega}.$
\par
(1)\ $\mathcal{A}(\omega)$ is a nonempty compact subset of $X$, and if $X$ is connected,
it is a connected subset of $K(\omega)$.
\par
(2)\ The family $\mathcal{A}(\omega),\ \omega\in \Omega$, is measurable.
\par
(3)\ $\mathcal{A}(\omega)$ is invariant in the sense that
$$S(t,s;\omega)\mathcal{A}(\vartheta_{s}\omega)= \mathcal{A}(\vartheta_{t}\omega),\ \ s\leq t.$$
\par
(4)\ It attracts all bounded sets from $-\infty$: for bounded $B\subset X$ and $\omega\in \tilde{\Omega}$
\begin{eqnarray*}
d(S(t,s;\omega)B, \mathcal{A}(\vartheta_{t}\omega))\rightarrow 0,\ \ when\ s\rightarrow -\infty.
\end{eqnarray*}
Moreover, it is the minimal closed set with this property: if $\tilde{\mathcal{A}}(\vartheta_{t}\omega)$ is a closed attracting set, then $\mathcal{A}(\vartheta_{t}\omega)\subset \tilde{\mathcal{A}}(\vartheta_{t}\omega).$
\par
(5)\ For any bounded set $ B\subset X,\ d(S(t,s;\omega)B, \mathcal{A}(\vartheta_{t}\omega))\rightarrow 0$ in probability when $t\rightarrow \infty.$\\
And if the time shift $\vartheta_{t},t\in \mathbb{R}$ is ergodic
\par
(6)\ there exists a bounded set $B\subset X$ such that
\begin{eqnarray*}
\mathcal{A}(\omega)= \mathcal{A}(B, \omega).
\end{eqnarray*}
\par
(7)\ $ \mathcal{A}(\omega)$ is the largest compact measurable set which is invariant in sense of Definition $2.5.$
\end{theorem}
In this article, to prove our main result we will use the compactness arguments of solution operator and the regularity of the solution to $(1.1)-(1.5)$ which rely on the two lemmas below. The first lemma is Aubin-Lions Lemma whose proof can be founded in $\cite{Au, Li}.$
\begin{lemma}
Let $B_{0}, B, B_{1}$ be Banach spaces such that $B_{0}, B_{1}$ are reflexive and $B_{0}\overset{c}{\subset}B\subset B_{1}.$ Define,
for $0<T<\infty,$
\begin{eqnarray*}
X:=\Big{\{}h \Big{|}h\in L^{2}([0,T]; B_{0}), \frac{dh}{dt}\in L^{2}([0,T]; B_{1})\Big{\}}.
\end{eqnarray*}
Then $X$ is a Banach space equipped with the norm $|h|_{L^{2}([0,T]; B_{0})}+|h'|_{L^{2}([0,T]; B_{1})}.$ Moreover, $$X \overset{c}\subset {L^{2}([0,T]; B)}.$$
\end{lemma}
The following lemma, a special case of a general result of Lions and Magenes $\cite{LM}$,  will help us to show the continuity of the solution to stochastic nematic liquid crystals with respect to time.
\begin{lemma}
Let $V , H, V'$ be three Hilbert spaces such that $V \subset H = H¡ä\subset V' $
, where$ H'$ and $V'$ are the dual spaces of $H$ and $V$ respectively. Suppose $u \in
L^{2}([0, T]; V )$ and $u'\in L^{2}([0, T]; V')$. Then $u$ is almost everywhere equal to a function
continuous from $[0, T]$ into $H$.
\end{lemma}

\section{Random attractor for the weak solution}
One of our main results in this article is to prove:
\begin{theorem}
Let $ \mathbf{v}_{0}\in \mathbf{H}, \mathbf{d}_{0}\in \mathbb{H}^{1}$ in $(1.4)$ and $f(\mathbf{d})$ is given by $(1.12)$. Assume $|h|\ll 1$ and $(2.14)$
hold. Then the solution operator $(S(t,s;\omega))_{t\geq s,\omega\in \tilde{\Omega}} $ of $(1.1)-(1.5): S(t,s;\omega)(\mathbf{v}_{s}, \mathbf{d}_{s})=(\mathbf{v}(t), \mathbf{d}(t) ) $ has properties $\mathrm{(i)}-\mathrm{(iv)}$ of Theorem 2.2 and possesses a compact absorbing ball $\mathcal{B}(0,\omega)$ in $\mathbf{H} \times \mathbb{H}^{1}$ at time $0.$  Furthermore, for $\tilde{\mathbb{P}}$-a.e. $\omega,$ the set
$$\mathcal{A}(\omega)=\overline{\bigcup_{B\subset \mathbf{H}\times \mathbb{H}^{1}}\mathcal{A}(B,\omega) } $$
where the union is taken over all the bounded subsets of $\mathbf{H}\times \mathbb{H}^{1}$ is the random attractor of $(1.1)-(1.5)$ and possesses the properties $(1)-(7)$ of Theorem $2.2$ with space $X$ replaced by space $\mathbf{H}\times \mathbb{H}^{1}.$
\end{theorem}
\begin{proof}
The results of this theorem follows directly from Proposition 3.6 and Theorem 2.2. \hspace{\fill}$\square$
\end{proof}
\par
 The rest of this section is to find a compact absorbing ball for $(1.1)-(1.5)$ in $\mathbf{H}\times \mathbb{H}^{1}$. We will achieve our goal by six steps. In subsection 3.1, we use $a$ $new$ $technique$ logarithmic energy estimates combined with It\^{o} formula to obtain the uniform a priori estimates in $(L^{4N+2}(D))^{3}$ which is very important to study the long-time behavior of the stochastic nematic liquid system (see Lemma 3.1 and the proof of Proposition 3.1). Then in  subsection 3.2, we obtain the absorbing ball for the solution to $(1.1)-(1.5)$ in the space $\mathbf{H}\times \mathbb{H}^{1}$ in Proposition 3.1. As the third step, in subsection 3.3, we use a new technique to prove the solution is indeed a stochastic flow in Proposition 3.2. In the next subsection, by Proposition 3.3 and Proposition 3.4 we verified  two a priori estimates of the Aubin-Lions lemma which is used to obtain a convergent subsequence of the solutions to $(1.1)-(1.5).$  In subsection 3.5, taking advantage of the convergent subsequence and the continuity of the solutions with respect to initial data in $\mathbf{H}\times \mathbb{H}^{1}$, we prove the solution operator $S(t,s; \omega)$ is almost surely compact from $\mathbf{H}\times \mathbb{H}^{1}$ to $\mathbf{H}\times \mathbb{H}^{1}$ for all $s,t\in \mathbb{R}, s<t\ (see\ Proposition\ 3.5).$ Finally, using the existence of absorbing ball and compactness of the solution operator, we construct a compact absorbing ball in Proposition 3.6.

\par
To study the long time behavior of $(1.1)-(1.5)$, we introduce a modified stochastic convolution. Let  $t\in \mathbb{R}$ and $\beta\in \mathbb{R}_{+}.$ For simplicity, we still define
\begin{eqnarray*}
\mathbf{z}(t):= \int_{-\infty}^{t}e^{-(t-s)(A_{1}+\beta)}dW_{1}(s)  .
\end{eqnarray*}
Then by $(2.15)$, we have $ \mathbf{z}(\omega)\in C([0,T]; \mathbf{H}^{2}), \mathbb{P}-a.e.$ and satisfies the linear equation
\begin{eqnarray*}
&&d\mathbf{z}=(-A_{1}\mathbf{z}-\beta \mathbf{z})dt+dW_{1},\\
&&\mathbf{z}(x,t)=0,\ \forall (x,t)\in \Gamma\times \mathbb{R},\\
&&\mathbf{z}(t_0)=\mathbf{z}_{t_0},
\end{eqnarray*}
where $\mathbf{z}_{t_0}=\int_{-\infty}^{t_0}e^{s(A_{1}+\beta)}dW(s).$
Let $ (\mathbf{v}_{t_{0}},\mathbf{d}_{t_{0}})\in \mathbf{H}\times \mathbb{H}^{1} ,$ then in view of Theorem 2.1, $( \mathbf{v},\mathbf{d})$ is the unique global weak solution to $(1.1)-(1.5)$ on $[t_{0}, \infty)$ with $\mathbf{v}(t_{0})= \mathbf{v}_{t_{0}}$ and $\mathbf{d}(t_{0})= \mathbf{d}_{t_{0}} .$
Making the classic change $(\mathbf{v},\mathbf{d})= (\mathbf{u}+\mathbf{z}, \mathbf{d})$, then $(\mathbf{u}, \mathbf{d})$ satisfies the following system $(3.16)-(3.20)$.
\begin{eqnarray}
d\mathbf{u}+((\mathbf{u}+\mathbf{z})\cdot \nabla (\mathbf{u}+\mathbf{z})+\nabla p-\Delta \mathbf{u})dt+\nabla\cdot(\nabla \mathbf{d}\odot \nabla \mathbf{d})dt&=&\beta \mathbf{z},\\
\nabla\cdot \mathbf{u}&=&0,\\
\mathbf{d}_{t}+[(\mathbf{u}+\mathbf{z})\cdot \nabla\mathbf{d}]- (\Delta \mathbf{d}- f(\mathbf{d}))-\frac{1}{2}(\mathbf{d}\times h)\times h&=&(\mathbf{d}\times h) \dot{W}_{2},\\
\mathbf{u}(x,t)=0,\  \frac{\partial\mathbf{d}(x,t)}{\partial \mathbf{n}}=0,\ \ (x,t)\in \Gamma &\times& [t_{0}, \infty), \\
\mathbf{u}|_{t=0}=\mathbf{v}_{t_0}(x)\ \mathrm{with}\ \nabla\cdot \mathbf{v}_{t_0}=0,\ \ \ \  \mathbf{d}|_{t=t_0}=\mathbf{d}_{t_0}(x),\ \ x&\in& \mathbf{D}.
\end{eqnarray}


\par
\noindent ${3.1.\ \mathbf{Absorbing}\ \mathbf{ball}\ \mathbf{of}\  \mathbf{d}\ \mathrm{in}\ ({L}^{4N+2}(\mathbf{D}))^{3}. } $
\begin{lemma}
Denote by $ \mathbf{d}(t,\omega; t_{0},\mathbf{d}_{0}) $ the weak solution to $(1.3)$ with $\mathbf{d}(t_{0})= \mathbf{d}_{0} $, then it is uniformly bounded w.r.t. time $t$ and initial time $t_{0}$ in $(L^{4N+2}(\mathbf{D} ))^{3}$ provided the initial data is bounded, i.e.
\begin{eqnarray*}
\sup\limits_{(t,t_{0})\in \{(t,t_{0})|t\in \mathbb{R}, t_{0}\leq t \}}| \mathbf{d}(t,\omega; t_{0},\mathbf{d}_{0})|_{4N+2}^{4N+2}<\infty.
\end{eqnarray*}
\end{lemma}
\begin{Rem}
The estimates of this lemma improved bounds for the solutions to $(1.1)-(1.3).$ These bounds are uniform with respect to time $t$ and  initial time $t_{0}.$ These estimates of the uniform boundedness improve bounds obtain in Proposition C.1 of $\cite{BHR1}$ and bounds in Proposition 5.4 of  $\cite{BMAA}.$ These uniform bounds also allow us to obtain the absorbing balls for the solution in various function spaces(see $(3.36)$). The radii of these absorbing balls are independent of the initial data. This opens the way for finding the random attractor in the weak solution space $\mathbf{H}\times \mathbb{H}^{1}$. Maybe
this lemma is a basic result to study the long-time behavior of stochastic nematic liquid crystals i.e. the existence of random attractor and ergodicity for this stochastic system.
\end{Rem}
\begin{proof}
Since $ |\mathbf{d}|_{4N+2}^{4N+2}$ is twice Fr\'{e}chet differentiable for $\mathbf{d}\in \mathbb{H}^{2}$, then  the first and second derivatives of $ |\mathbf{d}|_{4N+2}^{4N+2}$ are given by
\begin{eqnarray*}
(|\mathbf{d}|_{4N+2}^{4N+2})'[d_{1}]&=&(4N+2)\langle |\mathbf{d}|^{4N}\mathbf{d}, d_{1}\rangle,\\
(|\mathbf{d}|_{4N+2}^{4N+2})''[d_{1},d_{2}]&=&(4N+2)\langle |\mathbf{d}|^{4N}d_{1},  d_{2}\rangle+(4N+2)4N\langle |\mathbf{d}|^{2N-1} \mathbf{d}, d_{1}\rangle \langle |\mathbf{d}|^{2N-1} \mathbf{d}, d_{2}\rangle,
\end{eqnarray*}
where $d_{1},d_{2}\in\mathbb{ H}$. Hence
\begin{eqnarray*}
(|\mathbf{d}|_{4N+2}^{4N+2})'[\mathbf{d}\times h]=0 \ \ \ \mathrm{and}\ \ \  (|\mathbf{d}|_{4N+2}^{4N+2})'[(\mathbf{d}\times h) dW_{2}]=0.
\end{eqnarray*}
By elementary calculus, we have
\begin{eqnarray*}
\langle |\mathbf{d}|^{4N}\mathbf{d},  \mathbf{d}\times h\times h   \rangle=-\langle |\mathbf{d}|^{4N} \mathbf{d}\times h, \mathbf{d}\times h\rangle.
\end{eqnarray*}
Then, applying It\^{o} formula to $|\mathbf{d}|_{4N+2}^{4N+2}$ we have
\begin{eqnarray*}
|\mathbf{d}(t)|_{4N+2}^{4N+2}&=&|\mathbf{d}(t_{0})|_{4N+2}^{4N+2}-(4N+2)\int_{t_{0}}^{t}\langle |\mathbf{d}|^{4N}\mathbf{d}, -\Delta \mathbf{d}+(\mathbf{u+z})\cdot \nabla \mathbf{d}+f(\mathbf{d}) ds   \rangle\nonumber\\
&&+\frac{1}{2}\int_{t_{0}}^{t}(4N+2)\langle |\mathbf{d}|^{4N}\mathbf{d},  d\times h\times h   \rangle ds\nonumber\\
&&+\frac{1}{2}\int_{t_{0}}^{t}(4N+2)\langle |\mathbf{d}|^{4N} \mathbf{d}\times h, \mathbf{d}\times h\rangle ds\nonumber\\
&&+\frac{1}{2}\int_{t_{0}}^{t}(4N+2)4N\langle |\mathbf{d}|^{2N-1} \mathbf{d} , \mathbf{d}\times h\rangle ^{2}ds\nonumber\\
&=& |\mathbf{d}(t_{0})|_{4N+2}^{4N+2}-(4N+2)(4N+1)\int_{t_{0}}^{t}\int_{\mathbf{D}}|\mathbf{d}|^{4N}|\nabla \mathbf{d}|^{2}dxds\nonumber\\
&&-(4N+2)\int_{t_{0}}^{t}\langle |\mathbf{d}|^{4N}\mathbf{d}, f( \mathbf{d})    \rangle,
\end{eqnarray*}
which implies that
\begin{eqnarray}
d |\mathbf{d}(t)|_{4N+2}^{4N+2}+(4N+2)(4N+1)\int_{\mathbf{D}}|\mathbf{d}|^{4N}|\nabla \mathbf{d}|^{2}dx+(4N+2)\langle |\mathbf{d}|^{4N}\mathbf{d}, f( \mathbf{d})    \rangle=0.
\end{eqnarray}
By Young's inequality, for small positive constant $\varepsilon$ there exists a positive constant $c$ such that
\begin{eqnarray*}
|a_{k}||\mathbf{d}|^{2k+4N+2}\leq \frac{\varepsilon}{N}|\mathbf{d}|^{6N+2}+\frac{c}{N}|\mathbf{d}|^{4N+2},\ \ k=0, 1,\cdots, N-1.
\end{eqnarray*}
Therefore,
\begin{eqnarray}
\langle |\mathbf{d}|^{4N}\mathbf{d},  f(\mathbf{d})    \rangle&=& \sum_{k=0}^{N}a_{k}|\mathbf{d}|^{2k+4N+2}\nonumber\\
&\geq&-\varepsilon|\mathbf{d}|^{6N+2}-c|\mathbf{d}|^{4N+2}+a_{N}|\mathbf{d}|^{6N+2}.
\end{eqnarray}
Combing $(3.21)$ and $(3.22)$ yields
\begin{eqnarray}
d |\mathbf{d}(t)|_{4N+2}^{4N+2}+c|\mathbf{d}(t)|_{6N+2}^{6N+2}dt\leq c|\mathbf{d}(t)|_{4N+2}^{4N+2}dt,
\end{eqnarray}
which implies
\begin{eqnarray*}
d (|\mathbf{d}(t)|_{4N+2}^{4N+2}+1)+c(|\mathbf{d}(t)|_{6N+2}^{6N+2}+1)dt\leq c(|\mathbf{d}(t)|_{4N+2}^{4N+2}+1)dt,
\end{eqnarray*}
Diving $(|\mathbf{d}(t)|_{4N+2}^{4N+2}+1) $ on both sides yields,
\begin{eqnarray}
\frac{d}{dt}\ln(|\mathbf{d}(t)|_{4N+2}^{4N+2}+1 )+c(|\mathbf{d}(t)|_{4N+2}^{4N+2}dt+1)^{\frac{3N+1}{2N+1}}\leq c.
\end{eqnarray}
Since
\begin{eqnarray}
\ln(1+|x|)\leq 1+|x|\leq (1+|x| )^{\frac{3N+1}{2N+1}},\ \ \mathrm{for}\ \mathrm{all}\ x\in \mathbb{R},
\end{eqnarray}
By $(3.24)$ we have
\begin{eqnarray}
\frac{d}{dt}\ln(|\mathbf{d}(t)|_{4N+2}^{4N+2}+1 )+c\ln(|\mathbf{d}(t)|_{4N+2}^{4N+2}+1 )\leq c.
\end{eqnarray}
Let $y(t)=\ln(|\mathbf{d}(t)|_{4N+2}^{4N+2}+1 )$. Then multiplying $e^{ct}$ on both sides yields
\begin{eqnarray*}
\frac{d}{dt}(y(t)e^{ct})\leq ce^{ct},
\end{eqnarray*}
which implies
\begin{eqnarray}
y(t)\leq y(t_{0})e^{-c(t-t_{0})}+\int_{t_{0}}^{t}e^{-c(t-s)}ds.
\end{eqnarray}
By $(3.27)$, $y(t)$ in uniformly bounded with respect to time $t$ and initial time $t_{0}$ provided the initial data $y(t_{0})$ is bounded. Therefore, this in turn implies the uniform  boundedness of $\mathbf{d}$ in $(L^{4N+2}(\mathbf{D}))^{3}$ with respect to initial time $t_{0}$.
Furthermore, if we denote $(\mathbf{u}(t,\omega; t_{0},\mathbf{u}_{0} ), \mathbf{d}(t,\omega; t_{0},\mathbf{d}_{0})) $ the solution to $(3.16)-(3.20),$ with $\mathbf{u}(t_{0})=\mathbf{u}_{0} $ and $\mathbf{d}(t_{0})=\mathbf{d}_{0} $, we infer that
\begin{eqnarray}
|f(\mathbf{d}(t,\omega; t_{0},\mathbf{d}_{0}))|_{2}^{2}\ \mathrm{and} \int_{D}\tilde {F}(|\mathbf{d}(t,\omega; t_{0},\mathbf{d}_{0})|^{2})d x\ \mathrm{are}\ \mathrm{uniformly}\ \mathrm{bounded}\ \mathrm{w.}\ \mathrm{r.}\ \mathrm{t.}\ t\ \mathrm{and}\  t_{0}.
\end{eqnarray}
 \hspace{\fill}$\square$
\end{proof}
\noindent ${3.2.\ \mathbf{Absorbing}\ \mathbf{ball}\ \mathbf{of}\  (\mathbf{u}, \mathbf{d})\ \mathrm{in}\ \mathbf{H}\times \mathbb{H}^{1}. } $

\begin{proposition}
There exists an  absorbing ball for the weak solution $(\mathbf{v}, \mathbf{d})$ to $(1.1)-(1.5)$ at any time $t(\in \mathbb{R})$ in $\mathbf{H}\times \mathbb{H}^{1}.$
\end{proposition}
\begin{proof}
The first order and second order Fr\'{e}chet derivative of $|\nabla \mathbf{d}|_{2}^{2}+\int_{D}F(|d(x)|^{2})dx$ are given by
\begin{eqnarray}
(|\nabla \mathbf{d}|_{2}^{2}+\int_{D}\tilde{F}(|\mathbf{d}(x)|^{2})dx )' [\xi] =  2\langle \nabla \mathbf{d},  \nabla\xi     \rangle+2\langle f(\mathbf{d}), \xi    \rangle
\end{eqnarray}
and
\begin{eqnarray}
&&(|\nabla \mathbf{d}|_{2}^{2}+\int_{D}\tilde {F}(|d(x)|^{2})dx )'' [\xi,\eta] \nonumber\\
&&=  2\langle \nabla \xi, \nabla \eta     \rangle+2\int_{\mathbf{D}} \tilde{f}(|\mathbf{d}|^{2}) (\xi\cdot\eta)    dx+4\int_{\mathbf{D}}\tilde{f}'(\mathbf{d}) (\mathbf{d}\cdot\xi)  (\mathbf{d}\cdot\eta) dx.
\end{eqnarray}

Since $\langle  \Delta \mathbf{d}-f(\mathbf{d}),   \mathbf{d } \times h \rangle dW_{2}=0 $ and $\mathbf{d}\perp \mathbf{d}\times h$,
applying It\^{o} formula to $|\nabla \mathbf{d}|_{2}^{2} +\int_{\mathbf{D}} \tilde{F}(|\mathbf{d}(x)|^{2})dx, \mathbf{d}\in \mathbb{H}^{1},$ yields
\begin{eqnarray}
&&\frac{1}{2}\frac{d(|\nabla \mathbf{d}|_{2}^{2}+\int_{\mathbf{D}}\tilde{F}(|\mathbf{d}|^{2})dx  )}{dt}+|\Delta \mathbf{d}-f(\mathbf{d})|_{2}^{2}\nonumber\\
&=&\langle (\mathbf{u}+\mathbf{z})\cdot \nabla \mathbf{d}, \Delta \mathbf{d} -f(\mathbf{d}) \rangle-\frac{1}{2}\langle \mathbf{d}\times h\times h, \Delta \mathbf{d}-f(\mathbf{d})    \rangle\nonumber\\
&&+\frac{1}{2}|\nabla \mathbf{d}\times h|_{2}^{2}+\int_{\mathbf{D}}\tilde{f}'(\mathbf{d})|\mathbf{d} \cdot (\mathbf{d}\times h) |^{2}dx\nonumber\\
&&+\frac{1}{2}\int_{\mathbf{D}}\tilde{f}(|\mathbf{d}|^{2})|\mathbf{d}\times h|^{2}dx+\langle  \Delta \mathbf{d}-f(\mathbf{d}), \mathbf{d}\times h     \rangle dW_{2}\nonumber\\
&=&\langle (\mathbf{u}+\mathbf{z})\cdot \nabla \mathbf{d}, \Delta \mathbf{d}  \rangle-\frac{1}{2}\langle \mathbf{d}\times h\times h, \Delta \mathbf{d}-f(\mathbf{d})    \rangle\nonumber\\
&&+\frac{1}{2}|\nabla \mathbf{d}\times h|_{2}^{2}+\frac{1}{2}\int_{\mathbf{D}}\tilde{f}(|\mathbf{d}|^{2})|\mathbf{d}\times h|^{2}dx.
\end{eqnarray}
Taking inner product of $(3.16)$ with $\mathbf{u}$ in $\mathbf{H}$ yields,
\begin{eqnarray}
\frac{1}{2}\frac{d|\mathbf{u}|_{2}^{2}}{dt}+|\nabla \mathbf{u}|_{2}^{2}=\langle \mathbf{u}\cdot \nabla \mathbf{z}+ \mathbf{z}\cdot \nabla \mathbf{z}, \mathbf{u }\rangle -
\langle \nabla\cdot ( \nabla \mathbf{d} \odot \nabla \mathbf{d} ) ,     \mathbf{u}\rangle .
\end{eqnarray}
Since  by integration by parts and boundary condition $(3.16)$ ,
\begin{eqnarray*}
\langle \mathbf{u}\cdot \nabla \mathbf{d },    \Delta \mathbf{d}    \rangle &=& \int_{\mathbf{D}}\mathbf{u}^{i}\partial_{x_{i}}\mathbf{d}^{k}\partial_{x_{j}x_{j}}\mathbf{d}^{k}\\
&=&-\int_{\mathbf{D}}\partial_{x_{j}}\mathbf{u}^{i}\partial_{x_{i}}\mathbf{d}^{k}\partial_{x_{j}}\mathbf{d}^{k}d\mathbf{D}
-\int_{\mathbf{D}}\mathbf{u}^{i}\partial_{x_{i}x_{j}}\mathbf{d}^{k}\partial_{x_{j}}\mathbf{d}^{k}d\mathbf{D}\\
&=&-\int_{\mathbf{D}}\partial_{x_{j}}\mathbf{u}^{i}\partial_{x_{i}}\mathbf{d}^{k}\partial_{x_{j}}\mathbf{d}^{k}d\mathbf{D},
\end{eqnarray*}
and
\begin{eqnarray*}
-
\langle \nabla\cdot ( \nabla \mathbf{d} \odot \nabla \mathbf{d} ) ,     \mathbf{u}\rangle &=&-\int_{\mathbf{D}}\partial_{x_{i}}(\partial_{x_{i}} \mathbf{d}^{k}\partial_{x_{j}}\mathbf{d}^{k} )\mathbf{u}^{j}dD\\
&=&\int_{\mathbf{D}}\partial_{x_{i}} \mathbf{d}^{k}\partial_{x_{j}}\mathbf{d}^{k} \partial_{x_{i}}\mathbf{u}^{j}dD,
\end{eqnarray*}
combining  $(3.28)$ and $(3.29)$ together yields
\begin{eqnarray}
&&\frac{1}{2}\frac{d(|\mathbf{u}|_{2}^{2} +|\nabla \mathbf{d}|_{2}^{2}+\int_{\mathbf{D}}\tilde{F}(|\mathbf{d}|^{2})dx )}{dt}+|\nabla \mathbf{u}|_{2}^{2}+|\Delta \mathbf{d}-f(\mathbf{d})|_{2}^{2}\nonumber\\
&=&\langle \mathbf{u}\cdot \nabla \mathbf{z}+ \mathbf{z}\cdot \nabla \mathbf{z}, \mathbf{u }\rangle-\frac{1}{2} \langle  \mathbf{d}\times h\times h,     \Delta \mathbf{d} -f(\mathbf{d})  \rangle\nonumber\\
&&+\frac{1}{2}|\nabla \mathbf{d}\times h|_{2}^{2}+\frac{1}{2}\int_{\mathbf{D}}\tilde{f}(|\mathbf{d}|^{2})|\mathbf{d}\times h|^{2}dx+\langle \mathbf{z}\cdot \nabla \mathbf{d}, \Delta \mathbf{d}   \rangle\nonumber\\
&\leq & c|\nabla \mathbf{z}|_{\infty}|\mathbf{u}|_{2}^{2}+c|\mathbf{z}|_{\infty}^{2}|\mathbf{u}|_{2}^{2}+c|\nabla \mathbf{z}|_{2}^{2}+\varepsilon |\Delta \mathbf{d}-f(\mathbf{d})|_{2}^{2}+c|\mathbf{d}|_{2}^{2}\nonumber\\
&&+\frac{1}{2}|\nabla \mathbf{d}\times h|_{2}^{2}+\frac{1}{2}\int_{\mathbf{D}}\tilde{f}(|\mathbf{d}|^{2})|\mathbf{d}\times h|^{2}dx+\varepsilon|\Delta \mathbf{d}|_{2}^{2}+c|\mathbf{z}|_{\infty}^{2}|\nabla \mathbf{d}|_{2}^{2}.
\end{eqnarray}
By $(1.12)$ and $(3.28),$ we have
\begin{eqnarray}
 |\int_{\mathbf{D}}\tilde{f}(|\mathbf{d}|^{2})|\mathbf{d}\times h|^{2}dx|\leq c\ \   \mathrm{and}\ \  |\int_{\mathbf{D}}\tilde{F}(|\mathbf{d}|^{2})dx|\leq c
\end{eqnarray}
 for all time $t \in \mathbb{R}$ and initial time $ t_{0}(\leq t).$
Since by the  Poincar\'{e} inequality and the Minkovski inequality
\begin{eqnarray}
c |\nabla \mathbf{d}|_{2}^{2}\leq\frac{1}{2}|\Delta \mathbf{d}|_{2}^{2}\leq |\Delta \mathbf{d}- f(\mathbf{d})|_{2}^{2} + |f(\mathbf{d})|_{2}^{2},
\end{eqnarray}
then in view of the  Poincar\'{e} inequality, the H\"{o}lder inequality, the Sobolev imbedding theorem and $(3.33)-(3.35)$,  we have
\begin{eqnarray}
&&\frac{1}{2}\frac{d(| \mathbf{u}|_{2}^{2} +|\nabla \mathbf{d}|_{2}^{2}+\int_{\mathbf{D}}\tilde{F}(|\mathbf{d}|^{2})dx )}{dt}+c| \mathbf{u}|_{2}^{2}+c|\Delta \mathbf{d}|_{2}^{2}+c\int_{\mathbf{D}}\tilde{F}(|\mathbf{d}|^{2})dx \nonumber\\
&\leq&c |f(\mathbf{d})|_{2}^{2}+c\int_{\mathbf{D}}\tilde{F}(|\mathbf{d}|^{2})dx+\varepsilon |\Delta \mathbf{d}|_{2}^{2}+c|\mathbf{z}|_{\infty}^{2}|\nabla \mathbf{d}|_{2}^{2}\nonumber\\
&&+c+|\nabla \mathbf{z}|_{\infty}|\mathbf{u}|_{2}^{2}+|\mathbf{z}|_{\infty}^{2}|\mathbf{u}|_{2}^{2}+|\nabla \mathbf{z}|_{2}^{2}\nonumber\\
&\leq&c (\|\mathbf{z}\|_{1}+ \|\mathbf{z}\|_{1}^{2} )(|\mathbf{u}|_{2}^{2}+|\nabla \mathbf{d}|_{2}^{2} )+c+c\|\mathbf{z}\|_{1}^{2}.
\end{eqnarray}

\par
Let $g(t)=|\mathbf{u}(t)|_{2}^{2}+|\nabla \mathbf{d}(t)|_{2}^{2}+\int_{\mathbf{D}}\tilde{F}(|\mathbf{d}(t)|^{2})dx,$ with $t\in \mathbb{R}$. Then by $(3.33)$ we have
\begin{eqnarray*}
\frac{d}{dt}g(t)+c(1-\|\mathbf{z}\|_{1}^{2})g(t)\leq c+c\|\mathbf{z}\|_{1}^{2}.
\end{eqnarray*}
Therefore,
\begin{eqnarray*}
\frac{d\Big{(}g(t)e^{\int_{t_{0}}^{t} c (1- \|\mathbf{z}(s)\|_{1}^{2})ds }\Big{)}}{dt}
\leq c(1+\|\mathbf{z}\|_{1}^{2})e^{\int_{t_{0}}^{t}  c (1-\|\mathbf{z}(s)\|_{1}^{2})},
\end{eqnarray*}
which implies
\begin{eqnarray}
g(t)\leq g(t_{0})e^{-\int_{t_{0}}^{t}  c (1-\|\mathbf{z}(s)\|_{1}^{2}) ds}+c\int_{t_{0}}^{t}(\|\mathbf{z}\|_{1}^{2}+1)e^{-\int_{s}^{t}  c(1-\|\mathbf{z}(u)\|_{1}^{2})du }ds.
\end{eqnarray}
 Since, the process $\mathbf{z}(t)$ is stationary and ergodic and $\|\mathbf{z}(t)\|_{1}$ has polynomial growth when $t\rightarrow -\infty$, following the classic arguments (see \cite{CDF,DZ} and other references),
$(3.37)$ gives us the desired uniform estimate which yields an absorbing ball for $(\mathbf{u}, \mathbf{d})$ in $\mathbf{H}\times \mathbb{H}^{1}$.
Following the standard arguments, we can also show that for any constant $r$ and $t$ there exists $\int_{t-r}^{t}(|\nabla \mathbf{u}(s)|_{2}^{2}+|\Delta \mathbf{d}(s)|_{2}^{2})ds $ is uniformly bounded with respect to initial time $t_{0}\leq t-r.$ \hspace{\fill}$\square$
\end{proof}
\begin{Rem}
As we see from $(3.36)$ that the uniform bounds for $\tilde{F}(\mathbf{d})$ obtained in Lemma 3.1 play a vital role to obtain the absorbing ball for $(\mathbf{v}, \mathbf{d})$ in $\mathbf{H}\times \mathbb{H}^{1}.$
\end{Rem}
\noindent ${3.3.\ \mathbf{The}\ \mathbf{solution}\     \mathbf{(v, d)}\  \mathbf{to}\ (1.1)-(1.5)\ \mathbf{is}\ \mathbf{indeed}\  \mathbf{a}\ \mathbf{stochastic}\ \mathbf{flow}. } $
\begin{proposition}
Let $(\mathbf{v}_{0}, \mathbf{d}_{0}) \in \mathbf{H}\times \mathbb{H}^{1}$ . Assume $(2.14)$ hold. If $h=(h_{1}, h_{2}, h_{3}), t\in \mathbb{R},$ each $ h_{i}$ is a non zero constant, $i=1,2,3,$ then the weak solution $(\mathbf{v}, \mathbf{d})$ to $(1)-(5)$ is a stochastic flow.
\end{proposition}
\begin{proof}
To show $(\mathbf{v}, \mathbf{d})$ is a stochastic flow, it is sufficient to show $ \mathbf{d}$ is a stochastic flow.

\par
Let $\alpha(t)= e^{W_{2}(t)}h, t\in \mathbb{R}$. For $a$ and $b$ in $\mathbb{R}^{3},$  define the inner product between $a$ and $b$ by $a\cdot b$.  Applying It\^{o} formula to $\alpha(t)\cdot \mathbf{d}(t)$ yields,
\begin{eqnarray*}
d(\alpha(t)\cdot \mathbf{d}(t) )&=&(d\alpha(t))\cdot \mathbf{d}(t)+\alpha(t)\cdot (d\mathbf{d}(t))\\
&=& \alpha(t)\cdot \mathbf{d}(t)\circ dW_{2}(t)-((\mathbf{u}+\mathbf{z})\cdot \nabla \mathbf{d})\cdot \alpha(t) +(\Delta \mathbf{d}-f(\mathbf{d}) )\cdot \alpha(t)\\
&&+\Big{(}\frac{1}{2}(\mathbf{d}\times h)\times h +(\mathbf{d}\times h )\circ W_{2}\Big{)} \cdot \alpha(t)\\
&=& -[(\mathbf{u}+\mathbf{z})\cdot \nabla (\alpha(t)\cdot \mathbf{d})] \cdot I+ \Delta (\alpha(t)\mathbf{d})-\bar{f}(\alpha(t)\mathbf{d} )\\
&&+(\alpha(t)\cdot \mathbf{d}(t))\circ dW_{2}(t),
\end{eqnarray*}
where $I=(1,1,1)\in \mathbb{R}^{3}$ and $\bar{f}(\alpha\cdot\mathbf{d})=\sum_{k=0}^{N}a_{k}|\mathbf{d}|^{2k}\alpha\cdot\mathbf{d}. $

Let $\alpha \cdot \mathbf{d}= \bar{\mathbf{d} }.$ Defining the weak solution to $(3.38)-(3.40)$ similarly to Definition 2.1, $ \bar{\mathbf{d} } $  is the unique weak solution to $(3.38)-(3.40)$ for given $\mathbf{u},$ $\mathbf{z}$ and $\mathbf{d}.$
\begin{eqnarray}
\bar{\mathbf{d}}_{t}+[(\mathbf{u}+\mathbf{z})\cdot \nabla\bar{\mathbf{d}}]\cdot I- (\Delta \bar{\mathbf{d}}- \bar{f}(\bar{\mathbf{d}}))&=&\bar{\mathbf{d}} \circ\dot{W}_{2},\\
\bar{\mathbf{d}}(x,t)=\mathbf{d}_0(x),\ \ (x,t)\in \Gamma &\times& [t_{0}, \infty), \\
\bar{\mathbf{d}}|_{t=t_0}=\mathbf{d}_{0}(x)\cdot h e^{W_{2}(t_{0})},\ \ x&\in& \mathbf{D}.
\end{eqnarray}
 Obviously, $ \bar{\mathbf{d} } $ is a stochastic flow (see $\cite{CF}$).

 \par
 Let $\bar{h}=(h_{1}, -h_{2}, -h_{3}).$ Define $\alpha_{1}= e^{W_{2}(t)}\bar{h}$ and $\mathbf{d}_{1}=\alpha_{1} \cdot \mathbf{d}$. Then following the above steps we will show $\mathbf{d}_{1}$ is also a stochastic flow. Since $\mathbf{d}=(d_{1}, d_{2}, d_{3})$ and $d_{1}=h_{1}^{-1}e^{-W_{2}(t)}(\mathbf{d}+\mathbf{d}_{1} )$, $d_{1}$ is a stochastic flow.

 \par
 Similarly, we can show $d_{2}$ and $d_{3}$ is a stochastic flow.
  Then the conclusions of the proposition follows.
   \hspace{\fill}$\square$
   \end{proof}
\begin{Rem}
If some $h_{i}=0$ or all $h_{i}=0, i=1,2,3.$ Then we can use Proposition 3.2 to prove $\mathbf{d}$ is a stochastic flow. For example, if $h_{1}=0.$ Then in the equation of $d_{1} $ the first component of $\mathbf{d}$, the coefficient of the noise $W_{2}$  is the first component of $\mathbf{d}\times h$ which is a linear function of $d_{2}$ and $d_{3}$. Then the noise in the equation of orientation field $d_{1}$ can be regarded as an additive noise, so $d_{1}$ is a stochastic flow.
\end{Rem}

\noindent ${3.4.\ \mathbf{Two}\ \mathbf{a}\ \mathbf{priori}\   \mathbf{estimates}\   \mathbf{for}\ \mathbf{the} \ \mathbf{Aubin-Lions}\ \mathbf{Lemma}. } $
\par
To establish that the solution operator to $(1.1)-(1.5)$ is compact in $\mathbf{H}\times \mathbb{H}^{1} $,  the first step is to use the Aubin-Lions lemma to obtain a convergent subsequence of $(\mathbf{v}, \mathbf{d})$ or equivalently a convergent subsequence of $(\mathbf{u}, \mathbf{d})$ which converges almost everywhere with respect to time $t\in [s, T], s,T\in \mathbb{R}$ and $s<T,$ in $\mathbf{H}\times \mathbb{H}^{1}$. Then we have to verify the two a priori estimates of Aubin-Lions lemma. The first one is to obtain a priori estimates about $(\mathbf{u}, A_{2}^{\frac{1}{2}}\mathbf{d})$ in $L^{2}([s,T]; \mathbf{V}\times \mathbb{H}^{1})$, which is proved in the following Proposition 3.3. The other one is to obtain a priori estimates of $(\frac{d}{dt}\mathbf{u}, \frac{d}{dt}A_{2}^{\frac{1}{2}} \mathbf{d})$ in $L^{2}([s,T]; \mathbf{H}^{-1}\times \mathbb{H}^{-1})$, which is obtained in the Proposition 3.4.

\begin{proposition}
Let  $B=\{(\mathbf{v}_{0}, \mathbf{d}_{0})\in \mathbf{H}\times \mathbb{H}^{1}| |\mathbf{v}_{0}|_{2}+\|\mathbf{d}_{0}\|_{1} \leq M   \}$ for some positive constant $M,$  then $\mathbb{P}$-a.s.
\begin{eqnarray}
\sup\limits_{( \mathbf{v}_{0}, \mathbf{d}_{0})\in B}\Big{(}\int_{0}^{T}\|\mathbf{u}(t)\|_{1}^{2}dt+ \int_{0}^{T}\|\mathbf{d}(t)\|_{2}^{2}dt\Big{)}<\infty.
\end{eqnarray}
\end{proposition}
\begin{proof}
Proposition 3.3 follows directly from Lemma 3.1 and $(3.36)$. \hspace{\fill}$\square$
\end{proof}
\begin{proposition}
Let  $B=\{(\mathbf{v}_{0}, \mathbf{d}_{0})\in \mathbf{H}\times \mathbb{H}^{1}| |\mathbf{v}_{0}|_{2}+\|\mathbf{d}_{0}\|_{1} \leq M   \}$ for some positive constant $M,$  then $\mathbb{P}$-a.s.
\begin{eqnarray}
\sup\limits_{( \mathbf{v}_{0}, \mathbf{d}_{0})\in B}\Big{(}\int_{0}^{T}\|\frac{d}{dt}\mathbf{u}(t)\|_{-1}^{2}dt+ \int_{0}^{T}\|\frac{d}{dt}A_{2}^{\frac{1}{2}}\mathbf{d}(t)\|_{-1}^{2}dt\Big{)}<\infty.
\end{eqnarray}
\end{proposition}
\begin{proof}
Consider the following equation
\begin{eqnarray}
d \mathbf{z}_{1}-\Delta \mathbf{z}_{1}&=&(\mathbf{d}\times h   )dW_{2}\\
\frac{\partial \mathbf{z}_{1}(x,t)}{\partial \mathbf{n}}&=&0, \ \ \forall (x, t)\in \Gamma\times \mathbb{R}_{+},\\
\mathbf{z}_{1}(x)|_{t=0}&=&0,\ \ \ \forall x\in \mathbf{D}.
\end{eqnarray}
Obviously, the unique solution $\mathbf{z}_{1}(t)$ to $(3.39)-(3.41)$ satisfies $\mathbf{z}_{1}(t) \in C([0,T]; \mathbb{H}^{2})$.  Make the classical change $\mathbf{d}-\mathbf{z}_{1}=\theta$, then  $(3.18)$ is equivalent to the following equation
\begin{eqnarray}
\theta_{t}+[\mathbf{u}+\mathbf{z}_{1}\cdot \nabla (\theta+\mathbf{z}_{1})]-\Delta \theta+f(\theta +\mathbf{z}_{1})-\frac{1}{2}(\theta +\mathbf{z}_{1})\times h\times h=0,
\end{eqnarray}
where $\theta$ satisfies the following initial boundary conditions
\begin{eqnarray}
\frac{\partial \theta(x,t)}{\partial \mathbf{n}}&=&0, \ \ \forall (x, t)\in \Gamma\times \mathbb{R}_{+},\\
\theta(x)|_{t=0}&=&\mathbf{d}_{0}(x),\ \ \ \forall x\in \mathbf{D}.
\end{eqnarray}
For  $\eta\in\mathbb{ H}^{1},$ taking inner product of $(3.46)$ with $A_{2}^{\frac{1}{2}}\eta$ in $\mathbb{H}$ yields,
\begin{eqnarray}
&&\langle \frac{d A_{2}^{\frac{1}{2}}\theta}{dt},   \eta\rangle=\langle \frac{d\theta}{dt},    A_{2}^{\frac{1}{2}}\eta \rangle\nonumber\\
&&=\langle \Delta\theta ,   A_{2}^{\frac{1}{2}}\eta  \rangle -\langle (\mathbf{u}+\mathbf{z})\cdot \nabla(\theta+\mathbf{z}_{1}),    A_{2}^{\frac{1}{2}}\eta     \rangle-\langle f(\theta+\mathbf{z}_{1}),  A_{2}^{\frac{1}{2}}\eta  \rangle\nonumber\\
&&+\frac{1}{2}\langle  (\theta+\mathbf{z}_{1})\times h\times h,   A_{2}^{\frac{1}{2}}\eta \rangle\nonumber\\
&\leq& |\Delta\theta|_{2}\|\eta\|_{1}+|\mathbf{u}+\mathbf{z}|_{4}|\nabla\theta+\nabla \mathbf{z}_{1}|_{4}\|\eta\|_{1}+|f(\theta+\mathbf{z}_{1})|_{2}\|\eta\|_{1}\nonumber\\
&&+c|\theta+\mathbf{z}_{1}|_{2}\|\eta\|_{1}\nonumber\\
&\leq& |\Delta\theta|_{2}\|\eta\|_{1}+c|\mathbf{u}+\mathbf{z}|_{2}^{\frac{1}{2}}|\|\mathbf{u}+\mathbf{z}\|_{1}^{\frac{1}{2}}|\nabla\theta+\nabla \mathbf{z}_{1}|_{2}^{\frac{1}{2}}|\Delta\theta +\Delta \mathbf{z}_{1}|_{2}^{\frac{1}{2}}\|\eta\|_{1}\nonumber\\
&&+|f(\theta+\mathbf{z}_{1})|_{2}\|\eta\|_{1}+c|\theta+\mathbf{z}_{1}|_{2}^{2},
\end{eqnarray}
which together with the regularity of $\mathbf{z}, \mathbf{z}_{1}$ and Theorem 2.1 implies
\begin{eqnarray}
\frac{d A_{2}^{\frac{1}{2}}\theta}{dt}\ \mathrm{is}\ \mathrm{bounded}\ \mathrm{in}\  L^{2}([0,T];\mathbb{H}^{-1}).
\end{eqnarray}
For  $\xi\in  \mathbf{H}^{1},$ taking inner product of $(3.16)$ with $\xi$ in $\mathbf{H}$ yields,
\begin{eqnarray}
\langle \frac{d \mathbf{u}}{dt} ,  \xi \rangle&=&-\langle \mathbf{u}\cdot\nabla\mathbf{ u}, \xi  \rangle-\langle \mathbf{u}\cdot\nabla\mathbf{ z}, \xi  \rangle\nonumber\\
&&-\langle \mathbf{z}\cdot\nabla\mathbf{ u}, \xi  \rangle-\langle \mathbf{z}\cdot\nabla\mathbf{ z}, \xi  \rangle-\langle \Delta \mathbf{u}, \xi    \rangle+\int_{\mathbf{D}}\partial_{x_{i}}\mathbf{d}^{k}\partial_{x_{j}}\mathbf{d}^{k}\partial_{x_{i}}\xi^{j}.
\end{eqnarray}
By the incompressible property of the fluid (see $(1.2)$), the boundary condition $(1.4)$ and integration by parts we obtain
\begin{eqnarray}
-\langle \mathbf{u}\cdot\nabla\mathbf{ u}, \xi  \rangle=\langle \mathbf{u}\cdot\nabla \xi ,  \mathbf{ u}\rangle\leq |\nabla \xi|_{2}|\mathbf{u}|_{4}^{2}\leq c\|\xi\|_{1}|\mathbf{u}|_{2}\|\mathbf{u}\|_{1}
\end{eqnarray}
In view of $(3.51)$ and $(3.52)$ we have
\begin{eqnarray*}
\langle \frac{d \mathbf{u}}{dt} ,  \xi \rangle&\leq& c|\mathbf{u}|_{2}\|\mathbf{u}\|_{1}\|\xi\|_{1}+|\mathbf{u}|_{2}|\nabla \mathbf{z}|_{4}|\xi|_{4}+|\nabla \mathbf{u}|_{2}|\mathbf{z}|_{4}|\xi|_{4}\nonumber\\
&&+|\nabla \mathbf{z}|_{2}|\mathbf{z}|_{4}|\xi|_{4}+|\nabla \mathbf{u}|_{2}|\nabla\xi|_{2}+c|\nabla\mathbf{d}|_{4}^{2}\|\xi\|_{1}\nonumber\\
&\leq& c |\mathbf{u}|_{2}\| \mathbf{u}\|_{1}\|\xi\|_{1}+c|\mathbf{u}|_{2}\| \mathbf{z}\|_{2}\|\xi\|_{1}+c\| \mathbf{u}\|_{1}\|\mathbf{z}\|_{1}\|\xi\|_{1}\nonumber\\
&&+c\|\mathbf{z}\|_{1}^{2}\|\xi\|_{1}+c\| \mathbf{u}\|_{1}\|\xi\|_{1}+c\|\mathbf{d}\|_{1}\|\mathbf{d}\|_{2}\|\xi\|_{1},
\end{eqnarray*}
which implies
\begin{eqnarray}
|\frac{d \mathbf{u}}{dt} |_{\mathbf{H}^{-1}}^{2}\leq c |\mathbf{u}|_{2}^{2}\| \mathbf{u}\|_{1}^{2}+ \|\mathbf{u}\|_{1}^{2}\| \mathbf{z}\|_{2}^{2}+c\|\mathbf{z}\|_{1}^{2}+c\|\mathbf{u}\|_{1}^{2}+c\|\mathbf{d}\|_{1}^{2}\|\mathbf{d}\|_{2}^{2}.
\end{eqnarray}
where $|\cdot|_{\mathbf{H}^{-1}}$ is the norm of the Sobolev space $\mathbf{H}^{-1}$ which is the dual space of $\mathbf{V}.$
Then Theorem 2.1, $(2.15)$ and $(3.53)$ imply
\begin{eqnarray}
\frac{d \mathbf{u}}{dt}\ \mathrm{is} \ \mathrm{bounded}\ \mathrm{in}\  L^{2}([0,T];\mathbf{H}^{-1}).
\end{eqnarray}
Therefore, $(3.42)$ follows by $(3.50)$ and $(3.54).$ \hspace{\fill}$\square$
\end{proof}

\noindent ${3.5.\ \mathbf{The}\ \mathbf{solution}\ \mathbf{S(t,s;\omega)}\ \mathbf{operator}\   \mathbf{to}\   \mathbf{(1.1)-(1.5)}\ \mathbf{is} \ \mathbf{impact}\  \mathbf{in}\  \mathbf{H}\times \mathbb{H}^{1}. } $\\
By virtue of Proposition 3.3, Proposition 3.4 and Lemma 2.2 with $V=\mathbf{V}$ or $\mathbb{H}^{1}$, $H=\mathbf{H}$ or $\mathbb{H}$ and $V'= \mathbf{H}^{-1}$ or $\mathbb{H}^{-1}$ we infer that
\begin{corollary}
\begin{eqnarray*}
&& \mathbf{v}\in C([0,T];\mathbf{H}),\ \mathrm{and}\ \mathbf{d}\in C([0,T];\mathbb{H}^{1}),\ \mathrm{for}\ \mathrm{arbitrary}\ T>0,  \ a.s..
\end{eqnarray*}
\end{corollary}
Then we will use Aubin-Lions lemma and Corollary 3.1 to show the following compactness result for the solution operators to $(1.1)-(1.5)$.
\begin{proposition}
For $ \omega \in \Omega$ , $S(t,s;\omega)$ is compact from  $\mathbf{H}\times \mathbb{H}^{1}$ to  $\mathbf{H}\times \mathbb{H}^{1}$,  for all $s, t \in \mathbb{R}$ and $ s \leq t.$
\end{proposition}
\begin{proof}
In Proposition 3.1, we have obtained absorbing ball for $(S(t,s;\omega))_{t\geq s,\omega\in \Omega}$ at any time $t\in \mathbb{R}$. We denote by $B(s,r(\omega))$, the absorbing ball at time $s$ with center $0\in \mathbf{H}\times \mathbb{H}^{1}$ and radius $r(\omega)$.  Denote by $\mathcal{B}$  a bounded subset $\mathbf{H}\times \mathbb{H}^{1}$ and set $\mathcal{C}_{T}$ as a subset of the function space:
\begin{eqnarray*}
 \mathcal{C}_{T}:=\Big{\{} \Big{(}\mathbf{v},   A_{2}^{\frac{1}{2}}\mathbf{d}\Big{)}\Big{|}(\mathbf{v}(s),\mathbf{d}(s))\in \mathcal{B},
 (\mathbf{v}(t),\mathbf{d}(t))=S(t,s;\omega)(\mathbf{v}(s),\mathbf{d}(s)),t\in[s,T], s\leq T \Big{\}}.
\end{eqnarray*}
Since both  $\mathbf{H}^{1}\subset \mathbf{H}$ and $\mathbb{H}^{1}\subset \mathbb{H}$ are compact, $\mathbf{H}^{1}\times \mathbb{H}^{1}  \subset \mathbf{H}\times \mathbb{H}$ is also compact. For arbitrary $(\mathbf{v}(s), \mathbf{d}(s))\in \mathcal{B} ,$ by Proposition 3.3 and Proposition 3.4 we know
\begin{eqnarray*}
(\mathbf{u}, A_{2}^{\frac{1}{2}}\theta)\ \mathrm{is}\ \mathrm{bounded}\ \mathrm{in}\   L^{2}([s, T];\mathbf{H}^{1}\times \mathbb{H}^{1})
\end{eqnarray*}
and
\begin{eqnarray*}
(\partial_{t}\mathbf{u},  \partial_{t}A_{2}^{\frac{1}{2}}\theta )\ \mathrm{is}\ \mathrm{bounded}\ \mathrm{in}\    L^{2}([s, T];\mathbf{H}^{-1}\times \mathbb{H}^{-1}).
\end{eqnarray*}
Therefore, by Lemma 2.1 with
\[
B_{0}=\mathbf{H}^{1}\times \mathbb{H}^{1},\ \ B=\mathbf{H}\times \mathbb{H},\ \ B_{1}=\mathbf{H}^{-1}\times \mathbb{H}^{-1},
\]
$\mathcal{C}_{T}$ is compact in $L^{2}([s, T];\mathbf{H}\times \mathbb{H} ).$
In order to show that for any fixed $t\in (s, T],\omega\in \tilde{\Omega}, S(t,s;\omega) $ is a compact operator in $\mathbf{H}\times \mathbb{H}^{1} ,$ we take any bounded sequences
$\{(\mathbf{v}_{0,n}, \mathbf{d}_{0,n} ) \}_{n\in \mathbb{N}}\subset \mathcal{B}$ and we want to extract, for any fixed $t\in (s, T]$ and $\omega\in \Omega,$ a convergent subsequence from
$\{S(t,s;\omega)(\mathbf{v}_{0,n}, \mathbf{d}_{0,n}) \}$.
Since $\{(\mathbf{v}, A_{2}^{\frac{1}{2}}\mathbf{d}  ) \}\subset \mathcal{C}_{T} ,$ by Lemma 2.1, there is a function $(\mathbf{v}_{*},\mathbf{d}_{*})$:
\[
(\mathbf{v}_{*},\mathbf{d}_{*})\in  L^{2}([s, T];\mathbf{H}\times \mathbb{H}^{1}),
\]
and a subsequence of $\{S(t,s;\omega)(\mathbf{v}_{0,n}, \mathbf{d}_{0,n} ) \}_{n\in \mathbb{N}},$ still denoted by  $\{S(t,s;\omega)(\mathbf{v}_{0,n}, \mathbf{d}_{0,n} ) \}_{n\in \mathbb{N}}$  for simplicity, such that
\begin{eqnarray}
\lim\limits_{n\rightarrow \infty}\int_{s}^{T}\|S(t,s;\omega)(\mathbf{v}_{0,n}, \mathbf{d}_{0,n} )-(\mathbf{v}_{*}(t),\mathbf{d}_{*}(t))  \|_{\mathbf{H}\times \mathbb{H}^{1}}^{2}dt=0,
\end{eqnarray}
where $\|\cdot\|_{\mathbf{H}\times \mathbb{H}^{1}}$ denotes the norm of the product space $\mathbf{H}\times \mathbb{H}^{1}.$
By measure theory, convergence in mean square implies almost sure convergence of a subsequence. Therefore, it follows from $(3.55)$ that there exists a subsequence of $\{S(t,s;\omega)(\mathbf{v}_{0,n}, \mathbf{d}_{0,n} ) \}_{n\in \mathbb{N}}, $
still denoted by $\{S(t,s;\omega)(\mathbf{v}_{0,n}, \mathbf{d}_{0,n}) \}_{n\in \mathbb{N}}$ for simplicity , such that
\begin{eqnarray}
 \lim\limits_{n\rightarrow \infty}\|S(t,s;\omega)(\mathbf{v}_{0,n}, \mathbf{d}_{0,n} )-(\mathbf{v}_{*}(t),\mathbf{d}_{*}(t))  \|_{ \mathbf{H}\times \mathbb{H}^{1}}=0,\ \ a.e.\ t\in (s, T].
\end{eqnarray}
Fix any $t\in (s, T ],$ by $(3.56),$ we can select a $t_{1}\in (s, t)$ such that
\[
 \lim\limits_{n\rightarrow \infty}\|S(t_{1}, s, \omega)(\mathbf{v}_{0,n}, \mathbf{d}_{0,n} )-(\mathbf{v}_{*}(t_{1}),\mathbf{d}_{*}(t_{1}))  \|_{\mathbf{H}\times \mathbb{H}^{1}}=0.
\]
Then by the continuity of the map $S(t,t_{1};\omega) $ in $ \mathbf{H}\times \mathbb{H}^{1}$ with respect to initial value, we have
\begin{eqnarray*}
S(t,s;\omega)(\mathbf{v}_{0,n}, \mathbf{d}_{0,n})&=&S(t,t_{1};\omega)S(t_{1},s;\omega)(\mathbf{v}_{0,n}, \mathbf{d}_{0,n})\\
&&\rightarrow S(t,t_{1};\omega)(\mathbf{v}_{*}(t_{1}),\mathbf{d}_{*}(t_{1})),\ \ \ \mathrm{in}\ \mathbf{H}\times \mathbb{H}^{1}.
\end{eqnarray*}
Hence for any $t\in(s, T], \{S(t,s;\omega)(\mathbf{v}_{0,n}, \mathbf{d}_{0,n} ) \}_{n\in \mathbb{N}}$   contains a subsequence which is convergent $\mathrm{in}\ \mathbf{H}\times \mathbb{H}^{1},$ which implies that for any fixed $t\in (s, T],\omega\in \Omega, S(t,s;\omega) $ is a compact operator in $\mathbf{H}\times \mathbb{H}^{1}.$  \hspace{\fill}$\square$
\end{proof}
\noindent ${3.6.\ \mathbf{The}\ \mathbf{existence}\ \mathbf{of}\ \mathbf{compact}\   \mathbf{absorbing}\   \mathbf{ball}\ \mathbf{in}\  \mathbf{H}\times \mathbb{H}^{1}. } $
\begin{proposition}
There exists a compact absorbing ball at any time $t\in \mathbb{R}$ for the stochastic dynamical system $(1.1)-(1.5)$ in $ \mathbf{H}\times \mathbb{H}^{1} . $
\end{proposition}
\begin{proof}
Using the notations given in Proposition 3.5, for $s<t,$ let $\mathcal{B}(t,\omega)= \overline{S(t,s;\omega)B(s,r(\omega))}$ be the closed set of $ S(t,s;\omega)B(s,r(\omega))$ in $\mathbf{H}\times \mathbb{H}^{1},$ where $B(s, r(\omega))$ is the absorbing ball at time $s$ with center $0\in \mathbf{H}\times \mathbb{H}^{1}$ and radius $r(\omega).$ Then, by the above arguments, we know $\mathcal{B}(t,\omega)$ is a random compact set in $\mathbf{H}\times \mathbb{H}^{1}$ for each $\omega.$  More precisely,  $\mathcal{B}(t,\omega)$ is a compact absorbing set in $\mathbf{H}\times \mathbb{H}^{1}$ at time $t\in \mathbb{R}.$ Indeed, for $(\mathbf{v}_{0,n}, \mathbf{d}_{0,n})\in \mathcal{B}$ a bounded subset in $\mathbf{H}\times \mathbb{H}^{1} $ there exists $s(\mathcal{B})\in \mathbb{R}_{-}$ such that if $s_{0}\leq s(\mathcal{B})$ , we have
\begin{eqnarray*}
S(t,s_{0};\omega)(\mathbf{v}_{0,n},\mathbf{d}_{0,n})=S(t,s;\omega)S(s,s_{0};\omega)(\mathbf{v}_{0,n},\mathbf{d}_{0,n})\subset S(t,s;\omega) B(s,r(\omega))\subset \mathcal{B}(t,\omega) .
\end{eqnarray*}
\hspace{\fill}$\square$
\end{proof}
\begin{Rem}
By Proposition 4.5 in $\cite{CF}$,   the existence of a random attractor as
constructed in the proof of Theorem 3.1 implies the existence of an invariant
Markov measure $\mu_{\cdot}\in \mathscr{P}_{\Omega}(\mathbf{H}\times \mathbb{H}^{1} )$ for $\varphi$ (see Definition 4.1 in $\cite{CF}$). In view of $\cite{C} $ there exists an invariant measure $\mu$ for the Markovian semigroup $\mathbb{P}_{t}[f(S(t,0,x))]$ satisfing
\[
\mu(B)= \int_{\Omega}\mu_{\omega}(B)\mathbb{P}(d\omega)\ \ and\ \ \mu(\mathcal{A}(\omega))=1,\ \ \mathbb{P}-a.e.,
\]
where $B\in \mathscr{B}(\mathbf{H}\times \mathbb{H}^{1})$ which is a Borel $\sigma$-algebra on $\mathbf{H}\times \mathbb{H}^{1} $. If the invariant measure for $\mathbb{P}_{t}$ is unique, then the
invariant Markov measure $\mu_{\cdot}$ for $\varphi$   is unique and given by
\[
\mu_{\omega}=\lim\limits_{t\rightarrow \infty}\varphi (t,\vartheta_{-t}\omega  )\mu.
\]
\end{Rem}

$\mathbf{Acknowledgments}.$ We deeply appreciate the valuable discussions with Professor Zdzislaw Brze$\acute{z}$niak for many days.

\def\refname{ Bibliography}

\end{document}